\documentclass[10pt,twoside,a4paper,openright,final]{article}
\usepackage[top=4.5cm, left=38mm, right=38mm,bottom=4.5cm]{geometry}
\usepackage{amssymb,amsmath,amsfonts,amscd,amsthm,verbatim}
\newtheorem{lemma}{Lemma}[section]
\newtheorem{theorem}[lemma]{Theorem}

\newtheorem{proposition}[lemma]{Proposition}

\newtheorem{remark}[lemma]{Remark}
\theoremstyle{definition}
\newtheorem{example}[lemma]{Example}
\newtheorem{algorithm}[lemma]{Algorithm}

\DeclareMathOperator{\Fr}{Fr}
\DeclareMathOperator{\Nm}{Nm}
\DeclareMathOperator{\Nz}{DegNz}

\title{Non-vanishing forms in projective space over finite fields}
\author{Samuel Lundqvist}
\date{}                                          % Activate to display a given date or no date

\begin{document}
\maketitle

\begin{abstract}
We consider a subset of projective space over a finite field and give bounds 
on the minimal degree of a non-vanishing form with respect to this subset. We also give an algorithm to compute a non-vanishing form. 
\end{abstract}

\section{Introduction}
Let $\mathbb{X} = \{p_1,\ldots, p_m\}$ be a set of points in $ \mathbb{P}^n( \Bbbk)$, where $\Bbbk$ is a field.
We say that a form $f$ in $\Bbbk[x_0, \ldots, x_n]$ is non-vanishing with respect to $\mathbb{X}$ if
$f(p_i) \neq 0$ for all $i$. When $ \Bbbk$ is an infinite field, there is an infinite number of linear forms which are non-vanishing on $\mathbb{X}$. This is not always the case when $\Bbbk$ is a finite field. Consider $\mathbb{X} = \{(1:0), (0:1), (1:1)\} \subseteq \mathbb{P}^1(\mathbb{F}_2)$, where $\mathbb{F}_2$ denotes the finite field with two elements. There are three linear forms in 
$\mathbb{F}_2[x_0,x_1]$: $x_0, x_1$ and $x_0+x_1$. We have 
$$x_0((0:1)) =  x_1((1:0)) = (x_0 + x_1) ((1:1)) = 0.$$
Thus, there is no linear non-vanishing form with respect to $\mathbb{X}$.

When $\mathbb{X} \subseteq \mathbb{P}^n(\Bbbk)$, let 
$\Nz(\mathbb{X}) \geq 1$ denote the least degree of a non-vanishing 
form $f$ with respect to $\mathbb{X}$. Denote by  
$\mathbb{F}_q$ the field with $q$ elements. 
In this paper, we will give bounds on $\Nz(\mathbb{X})$ when $\mathbb{X} \subseteq \mathbb{P}^n(\mathbb{F}_q)$ and an algorithm to compute a non-vanishing form. 

In the language of commutative algebra, a form $f$ in $\Bbbk[x_0, \ldots, x_n]$ is non-vanishing with respect to a set of 
projective points $\mathbb{X}$ if and only if $[f]$ is a non-zero divisor in the quotient ring 
$\Bbbk[x_0, \ldots, x_n]/I(\mathbb{X})$, where $I(\mathbb{X})$ 
is the vanishing ideal with respect to $\mathbb{X}$ and $[f]$ denotes the equivalence class of $f$ in 
$\Bbbk[x_0, \ldots, x_n]/I(\mathbb{X})$. Hence,  
$\Nz(\mathbb{X})$ is the least degree of  a non-zero divisor in 
$\Bbbk[x_0, \ldots, x_n]/I(\mathbb{X})$.
 
In Proposition 3.2 in \cite{Kreuzer}, Kreuzer shows that when 
$\mathbb{X} \subseteq \mathbb{P}^n(\mathbb{F}_q)$ and $|\mathbb{X}| \leq q$, then $\Nz(\mathbb{X})= 1$ 
by using that an element $[f] \in \mathbb{F}_q[x_0, \ldots, x_n]/I(\mathbb{X})$ is a non-zero divisor if $f$ does not belong to the union of the associated primes of $I(\mathbb{X})$. 
The same result using the same method is given independently in \cite{Lundqvist}. In that paper, linear non-zero divisors play an important role for computing varieties over one-dimensional graded rings.
In this paper we obtain the result as a special case of Theorem \ref{thm:bounds}.

%problem has been studied The problem of finding a non-vanishing form came from an algorithm to compute the variety of an ideal of projective dimension zero \cite{Lundqvist}. A crucial point of the algorithm is to find a minimal degree of a non-zero divisor in the quotient ring $\Bbbk[x_0, \ldots, x_n]/I(\mathbb{X})$. 

The existence of a non-zero divisor can also be stated in terms of the union of finite subspaces of 
$\Bbbk[x_0,\ldots,x_n]$ as follows. Let $\mathbb{X} = \{p_1, \ldots, p_m\}$ 
and let $I(p_i)$ be the set of all homogeneous polynomials vanishing on $p_i$. Let
$A = \cup_i I(p_i)$ and let $d$ be the least positive degree such that 
$\Bbbk[x_0, \ldots, x_n]_d$ is not contained in $A$. 
Then $d = \Nz(\mathbb{X})$.

Finally, if $\mathbb{X}$ is the set of $\mathbb{F}_q$ rational points 
of a hypersurface, given by a form $f$ in $\mathbb{F}_q[x_0,\ldots,x_n]$, then 
$\Nz(\mathbb{X})$ is the least degree of a form $g$ such that the variety determined by the ideal $(f,g)$, has no $\mathbb{F}_q$ rational points.

The author would like to thank Torsten Ekedahl for providing useful suggestions and comments on the paper.

\section{Preliminaries}

The results that we will present in this paper relies on Warning's theorem (Satz 3 in \cite{Warning}), which states that if the equation $f = 0$,  where $f$ in 
$\mathbb{F}_q[x_1, \ldots, x_n]$ is an element of degree $d<n$, has a solution, then it
has at least $q^{n-d}$ solutions. We give the projective version of the result as a lemma.

\begin{lemma} \label{lemma:first}
Let $f$ be a non-constant form in $\mathbb{F}_q[x_0, \ldots, x_n]$ of degree $d<n+1$. Then there are at least $1 + q + \cdots + q^{n-d}$ solutions to $f=0$ in $\mathbb{P}^n(\mathbb{F}_q)$.
\end{lemma}
\begin{proof} By Warning's theorem, there are at least $q^{n-d+1}$ solutions in 
$\mathbb{F}_q^{n+1}$. 
Removing the trivial solution, we are left with at least $q^{n-d+1} - 1$ zeroes. Thus, 
the number of projective solutions is at least
$(q^{n-d+1}-1)/(q-1) = 1 + q + \cdots + q^{n-d}$. 
\end{proof}

The requirement on $d$ in Warning's theorem is sharp. Indeed, Lang (Theorem 1 in \cite{Lang}) gives a
 construction of a form of degree $n+1$ in $\mathbb{F}_q[x_0, \ldots, x_n]$ which is 
 non-vanishing with respect to $\mathbb{P}^n(\mathbb{F}_q)$ --- it is the norm of the 
 element 
$x_0e_0 + \cdots + x_ne_n$, where $\{e_0, \ldots, e_n\}$ is a basis for an extension 
$\mathbb{F}_{q^{n+1}}$  of $\mathbb{F}_q$ of degree $n+1$. 
Recall that the norm of an element $\alpha$ in $\mathbb{F}_{q^{n+1}}$ is defined as 
$$\Nm(\alpha) = \alpha \cdot \Fr_q(\alpha) \cdot \Fr^2_q(\alpha) \cdots \Fr^n_q(\alpha),$$
where $\Fr_q$ is the Frobenius map
$$\Fr_q: \mathbb{F}_{q^{n+1}} \to \mathbb{F}_{q^{n+1}}, \alpha \mapsto \alpha^{q}.$$

\begin{example} \label{ex:first}
Suppose that we want to find a quadratic non-vanishing form with respect to 
$\mathbb{P}^1(\mathbb{F}_2)$. Since $1+y+y^2$ is irreducible over 
$\mathbb{F}_2[y]$, a basis for the extension field $F_4$ of $F_2$ is $\{1,y\}$. We get 
$$\Nm(x_0 + x_1y) = (x_0 + x_1y) \cdot \Fr_2( x_0 + x_1y) = (x_0 + x_1y) \cdot 
(\Fr_2(x_0) + \Fr_2(x_1) \Fr_2(y))$$  $$=(x_0 + x_1y) \cdot (x_0 + x_1 (1+y)) 
= (x_0 + x_1y) (x_0 + x_1 + x_1y) = x_0^2 + x_0x_1 + x_1^2.$$
\end{example}

The following lemma gives us our first bound on the least degree of a non-vanishing form $f$ with respect to $\mathbb{X}$.

\begin{lemma} \label{lemma:lang}
Let $\mathbb{X} \subseteq \mathbb{P}^n(\mathbb{F}_q)$. Then $\Nz(X) \leq n+1$. 
\end{lemma}
\begin{proof}
This is just a restatement of Theorem 1 in \cite{Lang}. 
\end{proof}

Some words about the notation. 
By saying that a subset $\mathbb{Y}$ of $\mathbb{X}$ is isomorphic to $\mathbb{P}^d(\mathbb{F}_q)$, 
we will mean \emph{linearly} isomorphic. Hence, when
$\mathbb{X}$ contains an isomorphic copy of $\mathbb{P}^d(\mathbb{F}_q)$, we can choose coordinates so that 
$(0:\cdots:0:a_0:\cdots:a_d) \in \mathbb{X}$ for all $(a_0, \ldots, a_d) \in \mathbb{F}_q^{d+1} \setminus (0,\ldots,0)$. 
Likewise, when we write $\mathbb{X}\subseteq \mathbb{P}^n(\mathbb{F}_q) \setminus \mathbb{Y}$, with $\mathbb{Y} \cong \mathbb{P}^d(\mathbb{F}_q)$, we mean that there is a linearly 
isomorphic copy of $\mathbb{P}^d(\mathbb{F}_q)$ which has empty intersection with $\mathbb{X}$. In this situation it is possible to 
choose coordinates such that $(0:\cdots:0:a_0:\cdots:a_d) \notin \mathbb{X}$ for all $(a_0, \ldots, a_d) \in \mathbb{F}_q^{d+1} \setminus (0,\ldots,0).$

\section{Geometric descriptions}

In this section we give bounds on $\Nz(\mathbb{X})$ in terms of the geometric structure of $\mathbb{X}$.

\begin{lemma}  \label{lemma:geometriclower}
Let $\mathbb{Y} \subseteq \mathbb{X} \subseteq \mathbb{P}^n(\mathbb{F}_q)$ and
suppose that $\mathbb{Y} \cong \mathbb{P}^d(\mathbb{F}_q)$. Then 
$\Nz(\mathbb{X}) \geq d+1$. 
\end{lemma}  

\begin{proof}
Choose coordinates such that 
$$(0:\cdots:0:a_0:\cdots:a_{d}) \in \mathbb{X}$$
for all $(a_0,\ldots,a_d) \in \mathbb{F}_q^{d+1} \setminus (0,\ldots,0)$.
Now consider a form $f$ of degree $i<d+1$ with respect to these coordinates. Let 
$$g = f(0,\ldots,0,y_0,\ldots, y_d).$$ If $g=0$, then
$f(p) = 0$ for any point $p$ in $\mathbb{Y}$. Otherwise, $g$ is a form of degree
$i$. By Lemma \ref{lemma:first} $g = 0$ has at least one solution in $\mathbb{Y}$. Thus, in both cases, there is a point $p \in \mathbb{X}$ such that 
$f(p) = 0$. Hence $\Nz(\mathbb{X}) \geq d+1$.
\end{proof}

In the following example we consider a subset $\mathbb{X}$ of $\mathbb{P}^2(\mathbb{F}_3)$ for which it holds that $\Nz(\mathbb{X}) = 2$ although no isomorphic copy of $\mathbb{P}^1(\mathbb{F}_3)$ is contained in $\mathbb{X}$. Thus, Lemma \ref{lemma:geometriclower} is non-sharp.

\begin{example} \label{example:counter}
Let 
$$\mathbb{X} = \{(1:0:0), (0:1:0), (0:0:1), (1:1:0), (0:1:2), (1:0:1)\} \subset \mathbb{P}^2(\mathbb{F}_3).$$
It is an easy exercise to show that $\Nz(\mathbb{X}) = 2$. If there was an isomorphic copy $\mathbb{Y}$ of $\mathbb{P}^1(\mathbb{F}_3)$ contained in $\mathbb{X}$, then there would be a linear change $y_0,y_1,y_2$ of coordinates such that $\mathbb{Y}$ is the zero locus of the linear polynomial $y_0$. Thus, to show that there is no isomorphic copy of $\mathbb{P}^1(\mathbb{F}_3)$ in $\mathbb{X}$, it is enough to show that for an 
arbitrary linear form $f = a_0x_0 + a_1x_1+a_2x_2$, there are at most three points from $\mathbb{X}$ in the zero locus $Z(f)$ of $f$. 
Clearly, if two of the $a_i$'s are zero, then $|Z(f)| = 3$. If only one of the 
$a_i$'s is zero, then $|Z(f)| \leq 3$. For the remaining values of $(a_0:a_1:a_2)$ we have $|Z(x_0 + x_1 + x_2)| = 1, |Z(x_0+x_1+2x_2)| = 1$ and 
$|Z(x_0 + 2x_1 + 2x_2)| = 3$. 
\end{example}

To get a sharp version of Lemma \ref{lemma:geometriclower}, we have to 
put some extra requirements on the set $\mathbb{X}$. We need two lemmas before we can prove the sharp version in Proposition \ref{prop:description}.

\begin{lemma} \label{lemma:projectiononestep}
Let $\mathbb{X} \subset \mathbb{P}^{n}(\mathbb{F}_q)$ and suppose that 
$(1:0\cdots:0) \notin \mathbb{X}$. Let $\pi$ be the projection from 
$\mathbb{X}$ to
$\mathbb{P}^{n-1}(\mathbb{F}_q)$
defined by sending 
$(a_0:\cdots:a_n)$ to $(a_1:\cdots:a_n)$. Then 
$\Nz(\mathbb{X}) \leq \Nz(\pi(\mathbb{X})).$
\end{lemma}
\begin{proof}
Since $(1:0:\cdots:0) \notin \mathbb{X}$, the projection is well defined.
Let $f \in \mathbb{F}_q[x_1, \ldots, x_n]$ be any non-vanishing form on $\pi (\mathbb{X})$. Then the embedding of $f$ into  $\mathbb{F}_q[x_0, \ldots, x_n]$ gives a non-vanishing form with respect to $\mathbb{X}$. Hence $\Nz(\mathbb{X}) \leq \Nz(\pi(\mathbb{X})).$

\end{proof}

\begin{remark} \label{remark}
When $\mathbb{X} \subset \mathbb{P}^{n}(\mathbb{F}_q)$, i.e. there exists a point $p \in  \mathbb{P}^{n}(\mathbb{F}_q) \setminus \mathbb{X}$, it is always possible to choose coordinates such that $p = (1:0:\cdots:0)$, and hence, define a projection from $\mathbb{X}$ to 
$\mathbb{P}^{n-1}(\mathbb{F}_q)$ by omitting the first coordinate.
\end{remark}

\begin{lemma} \label{lemma:series}
 Let 
$\mathbb{X} \subseteq \mathbb{P}^n(\mathbb{F}_q)$ 
and suppose that there is a series of linear projections   
$\mathbb{X} \mapsto \mathbb{X}^{(n-1)} \subseteq \mathbb{P}^{n-1}(\mathbb{F}_q), 
\mathbb{X}^{(n-1)} \mapsto \mathbb{X}^{(n-2)} \subseteq \mathbb{P}^{n-2}(\mathbb{F}_q), \ldots,  
\mathbb{X}^{(i+1)} \mapsto \mathbb{X}^{(i)} \subseteq \mathbb{P}^{i}(\mathbb{F}_q)$, all 
defined by omitting the first coordinate after a suitable linear change of coordinates. Suppose
further that $\mathbb{X}^{(i)} \cong \mathbb{P}^{d}(\mathbb{F}_q)$.
Then $\Nz(\mathbb{X}) \leq d+1$.
\end{lemma}

\begin{proof} From succesive use of Lemma \ref{lemma:projectiononestep}, we obtain
that $\Nz(\mathbb{X}) \leq \Nz(\mathbb{X}^{(n-1)}) \leq \cdots \leq \Nz(\mathbb{X}^{(i)})$. 
We have  $\Nz(\mathbb{P}^{d}(\mathbb{X})) = d+1$, so the lemma follows by repeating the 
embedding argument from Lemma \ref{lemma:projectiononestep}.

\end{proof}

We now combine Lemma \ref{lemma:geometriclower} and Lemma \ref{lemma:series}.

\begin{proposition} \label{prop:description}
Let $\mathbb{X} \subseteq \mathbb{P}^n(\mathbb{F}_q)$ and let $d$ 
be the greatest integer for which there exists $\mathbb{Y} \subseteq \mathbb{X}$, 
with $\mathbb{Y} \cong \mathbb{P}^{d}(\mathbb{F}_q)$. 
Suppose that there is a series of linear projections   
$\mathbb{X} \mapsto \mathbb{X}^{(n-1)} \subseteq \mathbb{P}^{n-1}(\mathbb{F}_q), 
\mathbb{X}^{(n-1)} \mapsto \mathbb{X}^{(n-2)} \subseteq \mathbb{P}^{n-2}(\mathbb{F}_q), \ldots,  
\mathbb{X}^{(d+1)} \mapsto \mathbb{X}^{(d)} \cong \mathbb{Y}$ all 
defined by omitting the first coordinate after a suitable linear change of coordinates. 
Then $\Nz(\mathbb{X}) = d+1$.
\end{proposition}
\begin{proof}
By Lemma \ref{lemma:geometriclower}, $\Nz(\mathbb{X}) \geq d+1$ and by 
Lemma \ref{lemma:series},  $\Nz(\mathbb{X}) \leq d+1$. 

\end{proof}

We end the chapter by giving a "cutting out" description of $\Nz(\mathbb{X})$.

\begin{lemma} \label{lemma:geometricupper}
Let $\mathbb{X} \subseteq \mathbb{P}^n(\mathbb{F}_q) \setminus \mathbb{Y}$, where 
$\mathbb{Y} \cong \mathbb{P}^{d}(\mathbb{X})$. Then 
$\Nz(\mathbb{X})  \leq n-d$. 
\end{lemma}  

\begin{proof}

Choose coordinates such that
$$(0:\cdots:0:a_0:\cdots:a_{d}) \notin \mathbb{X}$$ for all
$(a_0,\ldots,a_d) \in \mathbb{F}_q^{d+1} \setminus (0,\ldots,0)$.

With respect to these coordinates, the map $\mathbb{X} \to \mathbb{P}^{n-d-1}(\mathbb{F}_q)$,
$(a_0:\cdots:a_n) \mapsto (a_0:\cdots:a_{n-d-1})$ is well defined.
By Lemma \ref{lemma:lang}, there is a non-vanishing form in $\mathbb{F}_q[x_0, \ldots, x_{n-d-1}]$ of degree $n-d$. This form is naturally
embedded into $\mathbb{F}_q[x_0, \ldots, x_{n}]$ and is non-vanishing on $\mathbb{X}$.

\end{proof}

%We pose the following question which is motivated by its analogue to Theorem \ref{thm:description}.

%\begin{question}
%Let $\mathbb{X} \subseteq \mathbb{P}^n(\mathbb{F}_q)$ and
%suppose that $d$ is the greatest integer such that there is a $\mathbb{Y} \cong \mathbb{P}^d$ with
%$\mathbb{X} \subseteq \mathbb{P}^n(\mathbb{F}_q) \setminus \mathbb{Y}$. Does it then hold that $\Nz(\mathbb{X})  = n-d$? 
%\end{question}

%We remark that the statement holds for $d=1$. 

\section{Bounds on a non-vanishing form in terms of the number of elements}

\begin{lemma} \label{lemma:upperbound}
Let $\mathbb{X} \subseteq  \mathbb{P}^n( \mathbb{F}_q)$ and let $d$ be the least integer such that
$|\mathbb{X}| \leq q + \cdots + q^d$. Then $\Nz(\mathbb{X}) \leq d$. 
\end{lemma}

\begin{proof}

If $d>n$, then $\Nz(\mathbb{X}) \leq d$ by Lemma \ref{lemma:lang}. Suppose instead that $d \leq n$. We claim that it is possible to construct a series of linear projections 
$\mathbb{X} \mapsto \mathbb{X}^{(n-1)} \subseteq \mathbb{P}^{n-1}(\mathbb{F}_q), 
\mathbb{X}^{(n-1)} \mapsto \mathbb{X}^{(n-2)} \subseteq \mathbb{P}^{n-2}(\mathbb{F}_q), \ldots,  
\mathbb{X}^{(d)} \mapsto \mathbb{X}^{(d-1)} \subseteq \mathbb{P}^{d-1}(\mathbb{F}_q)$, all defined by omitting the first coordinate after a suitable linear change of coordinates. Thus, after proving the claim, the lemma follows from Lemma \ref{lemma:series}. 

Since $|\mathbb{X}| < |\mathbb{P}^n( \mathbb{F}_q)|$ it follows from Remark \ref{remark} that it is possible to define a
projection $\mathbb{X} \to  \mathbb{X}^{(n-1)}$. If  $d > n-1$, (e.g. $d=n$), then
we are done. Else, we have 
$|\mathbb{X}^{(n-1)}| < |\mathbb{P}^{n-1}( \mathbb{F}_q)|$ 
and we can repeat the argument to finally obtain
$\mathbb{X}^{(d-1)}$, which proves the claim. 
\end{proof}

The upper bound is sharp in the sense that for any $n$ and any $d$, there is a set $\mathbb{X}$ where $d$ is the least integer such that $|\mathbb{X}| \leq q + \cdots + q^d$ and $\Nz(\mathbb{X}) = d$. Indeed, let $\mathbb{X}$ be the image of any linear embedding of $\mathbb{P}^{d-1}(\mathbb{F}_q)$ into $\mathbb{P}^n(\mathbb{F}_q)$. Then we can apply Proposition \ref{prop:description} with $\mathbb{Y} =  \mathbb{P}^{d-1}(\mathbb{F}_q)$, so
it follows  that 
$\Nz(\mathbb{X}) = d$. Finally, $d$ is easily verified to be the least integer such that $|\mathbb{X}| \leq q + \cdots + q^d$.

\begin{lemma} \label{lemma:lowerbound}
Let $\mathbb{X} \subseteq \mathbb{P}^n(\mathbb{F}_q)$.
If $q^{n-d+2} + \cdots + q^n < |\mathbb{X}|$, then $\Nz(\mathbb{X}) \geq d$.
\end{lemma}

\begin{proof}
We can suppose that $d \geq 2$. Let $f$ be a form of degree $d-1$. By Lemma \ref{lemma:first}, the number of projective solutions is at least
$1 + q + \cdots + q^{n-d+1}$. It follows that there are at most
$|\mathbb{P}^n( \mathbb{F}_q)| - (1 + q + \cdots + q^{n-d+1}) = q^{n-d+2} + \cdots + q^n$ points where $f$ is non-vanishing. 
By assumption, $q^{n-d+2}  + \cdots + q^n< |\mathbb{X}|$. Thus, there is a point $p \in \mathbb{X}$ such that $f(p)=0$. Hence $\Nz(\mathbb{X}) > d-1$. 

\end{proof}

The bound in Lemma \ref{lemma:lowerbound} is also sharp in the same sense as described above. To see this, consider the set $\mathbb{X} = \mathbb{P}^n(\mathbb{F}_q) \setminus \mathbb{Y}$,
where $\mathbb{Y}$ is any isomorphic copy of $\mathbb{P}^{n-d}(\mathbb{F}_q)$.
By Lemma \ref{lemma:geometricupper}, we have $\Nz(\mathbb{X}) \leq d$. But 
$|\mathbb{X}| = q^{n-d+1} + \cdots + q^n$, so 
$\Nz(\mathbb{X}) \geq d$ by Lemma \ref{lemma:lowerbound}. Hence $\Nz(\mathbb{X}) = d$.

We can state the following theorem.

\begin{theorem} \label{thm:bounds}
Let $\mathbb{X} \subseteq \mathbb{P}^{n}(\mathbb{F}_q)$. If $|\mathbb{X}| \leq q^n$, let $d_1 = 1$. Otherwise, let $d_1$ be the greatest integer such that $q^{n-d_1+2} + \cdots + q^n < |\mathbb{X}|$. Let $d_2$ be the least 
integer such that $|\mathbb{X}| \leq q + \cdots + q^{d_2}$. Then 

$$ d_1 \leq \Nz(\mathbb{X}) \leq d_2.$$

The bounds are sharp in the sense that for any $n$ and any $d$, there is a set $\mathbb{X}_1$ such that $\Nz(\mathbb{X}_1)$  assumes the lower bound, and a set $\mathbb{X}_2$ such that $\Nz(\mathbb{X}_2)$ assumes the upper bound.

\end{theorem}
\begin{proof}
The first part of the theorem follows from Lemma \ref{lemma:upperbound} and Lemma \ref{lemma:lowerbound}. The second part follows from the remarks  after 
Lemma \ref{lemma:upperbound} and Lemma \ref{lemma:lowerbound}.
\end{proof}

When $|\mathbb{X}| \leq q$, then $d_1=d_2=1$ and we obtain Kreuzer's Proposition 3.2 \cite{Kreuzer} as a special case of Theorem \ref{thm:bounds}.

\section{An algorithm to compute a non-vanishing form}

To find $\Nz(\mathbb{X})$ is by no means an easy computational task, so
in order to get a fast method, we should not require that the degree of the returned form is minimal. The algorithm that we present below use the ideas of Lemma \ref{lemma:upperbound} and the degree of the returned form is bounded by $d$, where $d$  is the least 
integer such that $|\mathbb{X}| \leq q + \cdots + q^d$.
\begin{algorithm} \label{alg1}
\hspace{1cm}

\begin{enumerate}

\item If $|\mathbb{X}| = |\mathbb{P}^n(\mathbb{F}_q)|$, return the norm form. 

\item Else, there is a point $p \in \mathbb{P}^n(\mathbb{F}_q) \setminus \mathbb{X}$. Change coordinates so that 
$p = (0:\cdots:0:1)$. Project $\mathbb{X}$ to $\mathbb{X}' \subseteq  
\mathbb{P}^{n-1}(\mathbb{F}_q)$ by omitting the last coordinate. Let $f$ be a form returned after performing step 1 with $n=n-1$ and $\mathbb{X}=\mathbb{X}'$. Return $f$ with respect to the original coordinates. 

\end{enumerate}
\end{algorithm}

Note that to actually construct a non-vanishing form, we have to compute a norm form. Thus, our method relies on finding an irreducible over
 $\mathbb{F}_q[y]$. We refer the reader to \cite{Adleman}, where algorithms to construct irreducibles are discussed.

\begin{example} \label{example:alg}
Consider $\mathbb{P}^3(\mathbb{F}_2)$ and the point set $\mathbb{X} = \{(1:1:1:0),(0:0:0:1), (1:0:0:0), (0:1:0:0), (0:0:1:0), (1:1:1:1)\}$ with respect to some coordinates $x_0, x_1, x_2, x_3$. We have $2^2+2=6$ and, hence, $1 \leq \Nz(\mathbb{X}) \leq 2$ by Theorem \ref{thm:bounds}. 

We will now perform Algorithm \ref{alg1} on $\mathbb{X}$. Pick the point $(0:0:1:1) \notin \mathbb{X}$. With respect to the linear change of coordinates
$y_0 = x_0$, $y_1 = x_1$, $y_2 = x_2+x_3$ and $y_3 =  x_3$, this point reads
$(0:0:0:1)$. Thus, with respect to the coordinates $y_0, y_1, y_2, y_3$, we get 
$\mathbb{X} = 
 \{(1:1:1:0), (0:0:1:1), (1:0:0:0), (0:1:0:0), (0:0:1:0), (1:1:0:1)\}$ and  
 $(0:0:0:1) \notin \mathbb{X}$.  

We project down to $ \mathbb{P}^2( \mathbb{F}_2)$ and get the points $\pi(\mathbb{X}) = \{(1:1:1),\allowbreak (0:0:1), (1:0:0), (0:1:0), (1:1:0)\}$. Notice that 
$\pi(0:0:1:1) = \pi(0:0:1:0)$. Now we are looking for a non-vanishing form with respect to these five points in $\mathbb{P}^2( \mathbb{F}_2)$. We notice that the point $p' =\allowbreak (1:0:1)$ is missing from $\pi(\mathbb{X})$, so we consider the linear change of coordinates $z_0 = y_0+y_2,z_1=y_1,z_2=y_2$ for which $p' = (0:0:1)$. Then we get 
$\pi(\mathbb{X}) = \{(0:1:1), (1:0:1), (1:0:0), (0:1:0), (1:1:0)\}$ with respect to these coordinates. We project down to $ \mathbb{P}^1( \mathbb{F}_2)$ to get the points
$\{(0:1), (1:0), (1:1)\}$. From Example \ref{ex:first} we know that $z_0^2 + z_1^2 + z_0z_1$ is non-vanishing. Hence $(y_0+y_2)^2  + (y_0+y_2)y_1 + y_1^2$ is non-vanishing on $\pi(\mathbb{X})$, and we get the following quadratic form which is non-vanishing on 
$\mathbb{X}$: 
$$(x_0+x_2+x_3)^2 + (x_0 + x_2+x_3)x_1 + x_1^2.$$

\end{example}
The non-vanishing form constructed by the algorithm in Example \ref{example:alg} is two, and is in fact equal to $\Nz(\mathbb{X})$. We can verify this by showing that 
there is an embedding of $\mathbb{P}^1(\mathbb{F}_2)$ in $\mathbb{X}$, since it
then follows by Lemma  \ref{lemma:geometriclower} that $\Nz(\mathbb{X}) \geq 2$. 
Indeed, $\{(1:1:1:0), (0:0:0:1), (1:1:1:1)\} \cong \mathbb{P}^1(\mathbb{F}_2)$, which  can be seen by changing coordinates to $x_0+x_2, x_1+x_2,x_2,x_3$.

%\center
%\institute{
\vspace{0.5cm}
\footnotesize{
\hspace{0.5cm} Department of Mathematics, Stockholm University,

\hspace{0.5cm} SE-106 91 Stockholm, Sweden

\hspace{0.5cm} email: samuel@math.su.se}
\end{document}